\documentclass[12pt]{amsart}
\usepackage{epsfig}
\usepackage{amssymb, amsmath}
\usepackage{amsfonts,graphics,epsfig}
\usepackage{mathtools}
\usepackage{mathrsfs}
\usepackage{pb-diagram, pb-xy}
\usepackage[all]{xy}
\usepackage{comment}
\usepackage{color}
\usepackage{tikz-cd}
\usepackage{setspace}
\usepackage[left=1.5in,right=1.5in,top=1.5in,bottom=1.5in]{geometry}
\usepackage[colorlinks=true, linkcolor=red, urlcolor=blue, citecolor=blue]{hyperref}
\usepackage{enumitem}

\newcommand{\mbR}{\mathbb{R}}

\newcommand{\mbZ}{\mathbb{Z}}
\newcommand{\mbQ}{\mathbb{Q}}

\def\mbN{\mathbb{N}}
\def\mbP{\mathbb{P}}

\def\mfD{\mathfrak{D}}
\newcommand{\<}{\leq}
\def\>{\geq}

\def\ve{\varepsilon}
\def\vphi{\varphi}

\def\subset{\subseteq}

\newcommand{\lrd}{\lfloor}
\newcommand{\rrd}{\rfloor}
\newcommand{\lru}{\lceil}
\newcommand{\rru}{\rceil}

\newcommand{\bir}{\dashrightarrow}
\def\mcO{\mathcal{O}}

\def\mfC{\mathfrak{C}}

\newtheorem{theorem}{Theorem}[section]
\newtheorem{lemma}[theorem]{Lemma}
\newtheorem{proposition}[theorem]{Proposition}
\newtheorem{corollary}[theorem]{Corollary}
\newtheorem{conjecture}[theorem]{Conjecture}

\theoremstyle{remark}
\newtheorem{remark}[theorem]{Remark}
\theoremstyle{definition}
\newtheorem{definition}[theorem]{Definition}
\theoremstyle{definition}

\numberwithin{equation}{section}

\def\div{\operatorname{div}}

\def\Supp{\operatorname{Supp}}

\def\chr{\operatorname{char}}
\def\Ex{\operatorname{Ex}}

\def\vol{\operatorname{vol}}

\def\coeff{\operatorname{coeff}}
\def\nt{\operatorname{nt}}

\def\cent{\operatorname{center}}

\author{Omprokash Das}
\address{School of Mathematics\\
Tata Institute of Fundamental Research\\
Homi Bhabha Road, Navy Nagar\\
Colaba, Mumbai 400005}
\email{omprokash@gmail.com}
\email{omdas@math.tifr.res.in}

\date{}

\begin{document}
\title[Pluricanonical maps of surfaces of log-general type]{Boundedness of log-pluricanonical maps for surfaces of log-general type in positive characteristic}

\maketitle

\begin{abstract}
	In this article we prove the following boundedness result: Fix a DCC set $I\subset [0, 1]$. Let $\mfD$ be the set of all log pairs $(X, \Delta)$ satisfying the following properties: (i) $X$ is a projective surface defined over an algebraically closed field, (ii) $(X, \Delta)$ is log canonical and the coefficients of $\Delta$ are in $I$, and (iii) $K_X+\Delta$ is big. Then there is a positive integer $N=N(I)$ depending only on the set $I$ such that the linear system $|\lrd m(K_X+\Delta)\rrd|$ defines a birational map onto its image for all $m\>N$ and $(X, \Delta)\in\mfD$.
\end{abstract}

\tableofcontents

\section{Introduction}
Pluricanonical system (which determines the Kodaira dimension) of a variety is one of the fundamental birational invariants used in the classification theory of algebraic varieties. So understanding pluricanonical maps is of great importance. If $X$ is a variety of general type, then by definition the pluricanonical map $\phi_{rK_X}\bir \mbP(H^0(X, \mcO_X(rK_X)))$ is birational (onto its image) for all sufficiently large $r$. It is a natural question to then ask if there is an integer $r_n$ such that $\phi_{rK_X}$ is birational for all $r\>r_n$, uniformly for all varieties of general type of dimension $n$. When $X$ is a smooth curve of genus $g\>2$, it is easy to see that $\phi_{rK_X}$ is birational for all $r\>3$. When $X$ is a smooth surface of general type and the characteristic of the ground field is $0$, Bombieri showed in \cite{Bom70} that $\phi_{rK_X}$ is birational for all $r\>5$. The same result was later proved in characteristic $p>0$ by Ekedahl in \cite{Eke88}. Starting with dimension $\>3$ this becomes a very hard problem to study, and several partial cases were known in characteristic $0$ due to \cite{Han85, Ben86, Mat86, Che01, Kol86, Luo94, Luo00}. In 2006, Hacon and McKernan \cite{HM06}, and independently Takayama \cite{Tak06} using ideas of Tsuji \cite{Tsu06} made a breakthrough on this problem in all dimensions $\>3$. They proved that for any fixed positive integer $n$, there is another positive integer $r_n$ depending only on $n$ such that $\phi_{rK_X}$ is birational for all $r\>r_n$ and for all smooth projective varieties $X$ of general type of dimension $n$. A similar result in positive characteristic is unknown even in dimension $3$.\\

On the other hand there is an analogous problem for log pairs with wider range of applications, it says the following:
\begin{conjecture}\label{con:birational-boundedness}
	Fix a positive integer $n$, a DCC set $I\subset [0, 1]\cap\mbQ$. Let $\mfD$ be a collection of log pairs satisfying the following properties:
	\begin{enumerate}
		\item X is a projective variety of dimension $n$ defined over an algebraically closed field,
		\item $(X, \Delta)$ is log canonical and the coefficients of $\Delta$ are contained in the set $I$, and
		\item $K_X+\Delta$ is big.
	\end{enumerate}
	Then there is a positive integer $N=N(n, I)$ depending only on $n$ and the set $I$ such that the linear system $|\lrd m(K_X+\Delta)\rrd|$ defines a birational map onto its image for all $m\>N$ and $(X, \Delta)\in\mfD$.
\end{conjecture}
In dimension $2$ and characteristic $0$ this is proved by Todorov in \cite[Corollary 6.1]{Tod10}. In general in all higher dimensions and in characteristic $0$ this is proved by Hacon, McKernan and Xu in their paper \cite[Theorem C]{HMX14} as a part of their inductive arguments in the proof of the ACC property for log canonical thresholds. In this article we prove this conjecture for surfaces in positive characteristic. We note that our proof is characteristic free. More specifically we prove the following:
	\begin{theorem}\label{thm:main-theorem}
		Fix a DCC set $I\subset [0, 1]\cap\mbQ$. Let $\mfD$ be the set of all pairs $(X, \Delta)$ satisfying the following properties:
		\begin{enumerate}
			\item $X$ is a projective surface defined over an algebraically closed field,
			\item $(X, \Delta)$ is log canonical and the coefficients of $\Delta$ are contained in $I$, and
			\item $K_X+\Delta$ is big.
		\end{enumerate}
	Then there exists a positive integer $N=N(I)$ depending only on the set $I$ such that the linear system $|\lrd m(K_X+\Delta)\rrd|$ defines a birational map onto its image for all $m\>N$ and $(X, \Delta)\in\mfD$.	
		
	\end{theorem}
Conjecture \ref{con:birational-boundedness} is closely related to the boundedness problem of stable pairs, which in positive characteristic is known in dimension $2$ due to \cite{Ale94, AM04} and \cite{HK19}. We note that our Theorem \ref{thm:main-theorem} is not a corollary of the main results of these three papers. However, we do use some of the tools and techniques developed in those papers.\\
In characteristic $0$, one of the main tools used to prove Theorem \ref{thm:main-theorem} in dimension $2$ and higher is the McKernana's `\emph{Covering family of tigers}' \cite{Mck02}, for example, it is used in the proofs of \cite[Theorem 6.1]{Tod10} and \cite[Theorem C]{HMX14}. However, McKernan's technique makes use of Nadel vanishing theorem and generic smoothness theorem, both of which are known to fail in positive characteristic. Our method avoid use of both of these two theorems.\\

\textbf{Idea of the proof:} First passing to an appropriate log resolution we reduce the problem to a log smooth klt pair $(X, \Delta)$. Next using a theorem from \cite{Ale94} we reduce the problem to the case where the set $I$ is a finite set given by $I=\{\frac{i}{k}\; :\; i=1, 2,\ldots, k-1\}$, where $k$ is a fixed constant independent of the boundary divisors $\Delta$. At this stage using an argument of Alexeev \cite{Ale94} we also prove that the number of components of $\Delta$ is uniformly bounded by some positive integer $M=M(I)$ which depends only on the set $I$. Then we run a $(K_X+\Delta)$-MMP and obtain a minimal model $(X', \Delta')$. Next we show that the number of exceptional divisors $E$ over $X'$ with discrepancy $a(E, X', \Delta')<0$ is bounded above by the same constant $M$. Then by a result from \cite{HK19} and \cite{Ale94} there exists a positive integer $N=N(I)$ depending only on the set $I$ such that $N(K_{X'}+\Delta')$ is Carier for all pairs $(X', \Delta')$. Then by another lemma from \cite{HK19} (which is an application of the effective Matsusaka theorem) it follows that there is a positive integer $m_0=m_0(I)$ depending only on the set $I$ such that the linear system $|\lrd m(K_{X'}+\Delta')\rrd|$ gives a birational map onto its image for all pairs $(X', \Delta')$. Pulling back this linear system onto $X$ gives our result.\\

\noindent
{\it Acknowledgement.} The author would like to thank Christopher Hacon for answering his questions.

\section{Preliminaries}
Throughout the paper we work over \textbf{algebraically closed} fields of arbitrary characteristic, i.e. $\chr p\>0$.

\begin{definition}\label{def:singularities}
Let $X$ be a normal variety and $\Delta$ a $\mbQ$-divisor on $X$. If the coefficients of $\Delta$ are contained in the interval $[0, 1]$, then $\Delta$ is called a \emph{boundary} divisor. By \emph{log pair} $(X, \Delta)$ we mean that $\Delta$ is a boundary divisor and $K_X+\Delta$ is $\mbQ$-Cartier. For a log pair $(X, \Delta)$ we define terminal, \emph{canonical, klt, plt, dlt} and \emph{log canonical} or \emph{lc} singularities as in \cite[Definition 2.8]{Kol13}. Fix a real number $\ve>0$, for the defintion of $\ve$-klt and $\ve$-lc see \cite[Definition 1.5]{Ale94}. By a \emph{log smooth} pair $(X, \Delta)$ we mean that $X$ is smooth and $\Delta$ has simple normal crossing support.	
\end{definition}

\begin{definition}\label{def:rounding}
Let $x$ be a real number. We define $\lrd x\rrd$ as the \emph{largest integer} $\<x$ and $\lru x\rru$ as the \emph{smallest integer} $\>x$. Note that every real number $x$ satisfies $0\<x-\lrd x\rrd<1$. For an $\mbR$-divisor $D=\sum_{i=1}^n a_iD_i$, we define $\lrd D\rrd:=\sum\lrd a_i\rrd D_i$ and $\lru D\rru:=\sum\lru a_i\rru D_i$. For $D=\sum_{i=1}^na_iD_i$, we also define $I=\{1, 2,\ldots, n\}$, $I^{=1}=\{i\in I\;:\; a_i=1\}$ and $I^{<1}=\{i\in I\; :\; a_i<1\}$. Then we define the divisors $D^{=1}$ (resp. $D^{<1}$) as $D^{=1}:=\sum_{i\in I^{=1}} D_i$ (resp. $D^{<1}:=\sum_{i\in I^{<1}}a_iD_i$). If the coefficients of $D$ are contained in the interval $[0, 1]$, then $D$ has a unique decomposition as $D=D^{<1}+D^{=1}$.
\end{definition}

\begin{remark}\label{rmk:rounding-property}
Note that if $x\in\mbR$ and $n\in\mbZ$, then from the definition of $\lrd x\rrd$ it follows that $x+n\>0$ if and only if $\lrd x\rrd+n\>0$.	
\end{remark}

\section{Lemmas and Propositions}
In this section we will collect some important and useful results which will be needed in the next section for proving the main theorem.\\ 

The following lemma and its corollary will be used in the poof of various results throughout the paper without reference.

\begin{lemma}\label{lem:projection-formula}
Let $f:X\to Y$ be a proper birational morphism between two normal varieties. Let $D$ be a $\mbQ$-Cartier $\mbQ$-divisor on $Y$. Then $f_*\mcO_X(\lrd mf^*D\rrd+E)\cong\mcO_Y(\lrd mD\rrd)$ for all integer $m\>1$ and effective exceptional divisor $E\>0$.	
\end{lemma}

\begin{proof}
	Since the question is local one the base, we may assume that $Y$ is an affine variety. Therefore it is enough to prove that $H^0(X, \mcO_X(\lrd mf^*D\rrd+E))\cong H^0(Y, \mcO_Y(\lrd mD\rrd))$ via $f^*$. To that end choose $f^*\vphi\in H^0(X, \mcO_X(\lrd mf^*D\rrd+E))$. Then $\lrd mf^*D\rrd+E+\div(f^*\vphi)\>0$. This implies that $mf^*D+E+\div(f^*\vphi)\>0$; pushing this forward by $f$ we get $mD+\div(\vphi)\>0$, hence $\lrd mD\rrd+\div(\vphi)\>0$ (see Remark \ref{rmk:rounding-property}), i.e. $\vphi\in H^0(Y, \mcO_Y(\lrd mD\rrd))$. For the other inclusion choose $\psi\in H^0(Y, \mcO_Y(\lrd mD\rrd))$. Then $\lrd mD\rrd+\div(\psi)\>0$, and thus $mD+\div(\psi)\>0$. Pulling it back by $f$ we get $mf^*D+\div(f^*\psi)\>0$, and hence $\lrd mf^*D\rrd+E+\div(f^*\psi)\>0$, since $E$ is effective. Therefore $f^*\psi\in H^0(X, \mcO_X(\lrd mf^*D\rrd+E))$ and we are done.\\
\end{proof}

\begin{corollary}\label{lem:transfer-of-pluricanonical-maps}
Let $(X, \Delta)$ be a log canonical pair of dimension $2$ and $K_X+\Delta$ is a $\mbQ$-Cartier big divisor. We run a $(K_X+\Delta)$-MMP and end with a minimal model $(X', \Delta')$. If the linear system  $|\lrd m(K_{X'}+\Delta')\rrd|$ gives a birational map onto its image for some $m\>1$, then $|\lrd m(K_{X}+\Delta)\rrd|$ also gives a birational map.
\end{corollary}

\begin{proof}
	Let $f:X\to X'$ be the birational morphism induced by the MMP. Then applying the negativity lemma at each step of this minimal model program it is easy to see that we have
	\[
		K_X+\Delta=f^*(K_{X'}+\Delta')+\sum a_iE_i,
	\]
	where $a_i\>0$ for all $i$.\\
	Therefore $H^0(X, \mcO_X(\lrd m(K_X+\Delta)\rrd))=H^0(X', \mcO_{X'}(\lrd m(K_{X'}+\Delta')\rrd))$ for all $m\>1$ by Lemma \ref{lem:projection-formula}, and the result follows.
	
\end{proof}

In the following we will recall an important result of Alexeev from \cite{Ale94}. To make the statement of his theorem more precise we define some notation and terminologies first. 
\begin{definition}\label{def:redundant-components}
	Let $(X, B\>0)$ be a log canonical pair and $K_X+B$ a $\mbQ$-Cartier big divisor. We call a divisor $\phi(B)$ a \emph{redundant part} of $B$ if it satisfies the following properties:
	\begin{enumerate}[label=(\roman*)]
		\item $0\<\phi(B)\<B$ and for any prime Weil divisor $E$ contained in $\Supp(\phi(B))$, $\coeff_{\phi(B)}(E)=\coeff_{B}(E)$,
		\item $K_X+(B-\phi(B))$ is big, and
		\item $\phi(B)$ is a maximal divisor satisfying these conditions.
	\end{enumerate}
	The components of $B-\phi(B)$ are called the \emph{non-redundant} components of $B$.
\end{definition}

\begin{remark}\label{rmk:non-uniqueness-of-redundant-components}
	Note that $\phi(B)$ not unique in general, as there could be a different set of components $\phi'(B)$ of $B$ removing which could give $K_X+B-\phi'(B)$ big as well.
\end{remark}

The following important result due to Alexeev shows that under certain conditions the number of non-redundant components of the boundary divisor $B$ is bounded from above.	
	\begin{theorem}\cite[Theorem 7.3, Corollary 7.4]{Ale94}\label{thm:bounding-components}
		Fix a positive real number $\ve>0$ and a DCC set $I\subset [0, 1]\cap\mbQ$. Let $\mfC$ be a collection of pairs $(X, \Delta)$ satisfying the following properties:
		\begin{enumerate}
			\item $X$ is a projective surface,
			\item $(X, \Delta)$ is $\ve$-log canonical and the coefficients of $\Delta$ are contained in $I$, and
			\item $K_X+\Delta$ is big. 
		\end{enumerate}
		Furthermore, for a $\mbQ$-divisor $D$ on $X$ let $\nt(D)$ denote the number of irreducible components of $D$.\\
		 Then there exists a positive integer $A=A(I, \ve)>0$ depending only on the set $I$ and $\ve$ such that
		 \[
			\nt(\Delta-\phi(\Delta))\<A\quad\mbox{for all } (X, \Delta)\in\mfC\mbox{ and for all choices of } \phi(\Delta).
		\]
	\end{theorem}
\begin{proof}
	It follows from Theorem 7.3 and Corollary 7.4 of \cite{Ale94} and noting the fact that the DCC set $I\setminus\{0\}$ has a minimum.\\
\end{proof}

The next result from a recent paper of Hacon and Kov\'acs \cite{HK19} will play a crucial role in our proof of the main theorem. We note that the proof of this lemma follows as an easy consequence of effective Matsusaka theorem due to \cite{Ter99} and \cite{DE15} as explained in \cite{HK19}.
\begin{lemma}\cite[Corollary 1.14]{HK19}\label{lem:birational-bound}
	Let $X$ be a normal surface and $D$ a nef and big Cartier divisor. If $D^2\>\vol(K_X)$, then the linear system $|K_X+qD|$ defines a birational morphism onto its image for all $q\>18$.
\end{lemma}

In the next two results we will bound the number of exceptional divisors of negative discrepancies over a (log) minimal model $(X, \Delta)$ by the number of components of the boundary divisor $\Delta$ and also the Cartier index of $(K_X+\Delta)$, when $(X, \Delta)$ has $\ve$-klt singularities.
\begin{proposition}\label{pro:bounding-exceptional-divisors}
	Fix a positive integer $N>0$. Let $\mfC$ be the collection of pairs $(X, \Delta)$ satisfying the following properties:
	\begin{enumerate}
		\item $X$ is a projective surface,
		\item $(X, \Delta)$ has terminal singularities, and
		\item $\nt(\Delta)\<N$, i.e., the number of components of $\Delta$ is bounded by $N$.
	\end{enumerate}
	We run a $(K_X+\Delta)$-MMP and assume that $(X', \Delta')$ is the corresponding minimal model. Let $\mfC'$ be the collection of all such minimal models $(X', \Delta')$ for all $(X, \Delta)\in\mfC$. Then 
	\[
		\#\left\{E\; :\; E\mbox{ is exceptional over } X' \mbox{ with } a(E, X', \Delta')<0  \right\}\<N
	\] 
	for all $(X', \Delta')\in\mfC'$.
\end{proposition}

\begin{proof}
	Let $f:X\to X'$ be the birational morphism induced by the MMP. Then we have 
	\begin{equation}\label{eqn:log-equation-for-minimal-model}
		K_X+\Delta=f^*(K_{X'}+\Delta')+\sum a_iE_i,
	\end{equation}
	where $a_i\>0$ for all $i$.\\
	
 Let $F$ be an exceptional divisor over $X'$ with discrepancy $a(F, X', \Delta')=b<0$. Set $E:=\sum a_iE_i$, then $(X, \Delta-E)$ is terminal, since $(X, \Delta)$ is terminal. Therefore from $a(F, X, \Delta-E)=a(F, X', \Delta')=b<0$ it follows that $\mbox{center}_X (F)$ must be a component of $\Delta-E$. Now since the components of $E$ have non-negative discrepancies with respect to the pair $(X', \Delta')$ and $a(F, X', \Delta')<0$, it follows that $\mbox{center}_X(F)$ must be a component of $\Delta$. Finally, since the number of components of $\Delta$ is bounded above by $N$ for all $(X, \Delta)\in\mfC$, the required bound holds.

\end{proof}

\begin{theorem}\cite[Lemma 2.6]{HK19}\label{thm:bounding-Cartier-index}
	Fix a positive integer $k>0$ and a positive real number $\ve>0$. Let $\mfD$ be the set of all of pair $(X, \Delta)$ satisfying the following properties:
	\begin{enumerate}
		\item $X$ is a projective surface,
		\item $(X, \Delta\>0)$ has $\ve$-klt singularities, and
		\item the number of exceptional divisor $E$ over $X$ with $a(E, X, \Delta)<0$ is at most $k$.
	\end{enumerate}
	Then there exists a positive integer $N=N(k, \ve)$ depending only on $k$ and $\ve$ such that $NK_X$ is Cartier and $ND$ is also Cartier for any integral Weil divisor $D$ contained in the support of $\Delta$.\\
\end{theorem}

The following two technical results will be useful in the proof of the main theorem.
\begin{lemma}\label{lem:modifying-divisor}
Let $I\subset [0, 1]$ be a DCC set and $\delta$ is a real number satisfying $0<\delta<1$. Let $a>0$ be the minimum of the set $I\setminus\{0\}$. Set $k=\lru\frac{1}{a\delta}\rru$ and define $a_i':=\frac{\lrd ka_i\rrd}{k}$ for $a_i\in I\setminus\{0\}$. Then 
\[
(1-\delta)a_i<a_i'\<a_i.	
\] 
\end{lemma}

\begin{proof}
It is clear from the defintion of $a_i'$ that $a_i'\<a_i$, so we only need to prove the other inequality. For that it is enough to show that $(ka_i-\lrd ka_i\rrd)<ka_i\delta$. To that end observe that $k=\lru\frac{1}{a\delta}\rru\>\frac{1}{a_i\delta}$, since $a_i\>a$. Thus $ka_i\delta\>1>(ka_i-\lrd ka_i\rrd)$.
\end{proof}

\begin{lemma}\label{lem:terminalization}
Fix a positive integer $k$.	Let $(X, \Delta)$ be a log smooth klt pair of dimension $2$ with coefficients of $\Delta$ in the finite set $J=\{\frac{\ell}{k}\; :\; \ell=1, 2,\ldots, k-1\}$. Then there exists a crepant log resolution $f:X'\to X$ of the pair $(X, \Delta)$ such that $K_{X'}+\Delta'=f^*(K_X+\Delta)$, $(X', \Delta')$ has terminal singularities and the coefficients of $\Delta'$ are contained in the set $J$.
\end{lemma}

\begin{proof}
	Since $(X, \Delta)$ is a klt pair, by \cite[Proposition 2.36(2)]{KM98} there are finitely many exceptional divisors over $X$ with non-positive discrepancies. We will extract these divisors. Note that since $(X, \Delta)$ is a log smooth klt pair of dimension $2$, if $E$ is an exceptional divisor over $X$ with $a(E, X, \Delta)\<0$, then the $\mbox{center}_X (E)$ is a point on $X$ contained in the intersection of precisely two components of $\Delta$.

Now write $\Delta=\sum_{i=1}^N a_iD_i$. We claim that if $F$ is an exceptional divisor over $X$ such that $a(F, X, \Delta)\<0$ and $\mbox{center}_X(F)\in D_i\cap D_j$, then $a_i+a_j-1\>0$. To the contrary assume that $a_i+a_j-1<0$. Let $f_1:X_1\to X$ be the blow up at $\mbox{center}_X(F)$, $F_1$ is the exceptional divisor and $K_{X_1}+\Delta_1=f_1^*(K_X+
\Delta)$. Then $a(F_1, X, \Delta)=(1-a_i-a_j)>0$. If $\cent_{X_1}(F)=F_1$, then $a(F, X, \Delta)=a(F_1, X_1, \Delta_1)>0$ and we have a contradiction. If not, then $\cent_{X_1}(F)$ is a point contained in the support of $F_1$.  
Let $f_2:X_2\to X_1$ be the blow up at $\cent_{X_1}(F)$, $F_2$ is the exceptional divisor and $K_{X_2}+\Delta_2=f_2^*(K_{X_1}+\Delta_1)$. Then by Lemma \ref{lem:positive-dicrepancy} we have $a(F_2, X, \Delta)=a(F_2, X_2, \Delta_2)>0$. Thus if $\cent_{X_2}(F)=F_2$, then $a(F, X, \Delta)=a(F_2, X,\Delta)>0$ and we again have a contradiction, otherwise $\cent_{X_2}(F)$ is a point and we blow up $X_2$ at this point.
 Continuing this process, by \cite[Lemma 2.45]{KM98} after finitely many steps we arrive at a morphism $f_n:X_n\to X_{n-1}$ for $n\>1$, such that $\cent_{X_n}(F)=F_n=\Ex(f_n)$ and $a(F, X, \Delta)=a(F_n, X, \Delta)>0$. This is a contradiction.\\
Thus in order to extract the exceptional divisors over $X$ with non-positive discrepancies we only need to blow up the points in $D_i\cap D_j\neq\emptyset$ whenever $a_i+a_j-1\>0$. Let $g_1:Y_1\to X$ be the blow up of all the points of $D_i\cap D_j$ for all $i, j\in\{1,\ldots, N\}, i\neq j$, whenever $a_i+a_j-1\>0$. Write $K_{Y_1}+\Delta_1=g_1^*(K_X+\Delta)$ and let $E_1$ be a $g_1$-exceptional divisor whose coefficient in $\Delta_1$ is not zero. Then the coefficient of $E_1$ in $\Delta_1$ is of the form $a_i-a_j-1>0$, and it is easy to see that $a_i-a_j-1\in J$, since $a_i, a_j\in J$. Now observe that $(Y_1, \Delta_1)$ is a log smooth pair with coefficients of $\Delta_1$ contained in $J$. Suppose $\Delta_1=\sum_{i=1}^{N_1}a_{i1}D_{i 1}$, where $a_{i1}\in J$ for all $i\in\{1, \ldots, N_1\}$. Now let $g_2:Y_2\to Y_1$ be the blow up of all the points in $D_{i1}\cap D_{j 1}$ for all $i, j\in\{1,\ldots, N_1\}, i\neq j$, whenever $a_{i1}+a_{j1}-1\>0$. Write $K_{Y_2}+\Delta_2=g_2^*(K_{X_1}+\Delta_1)$. Then again as before we see that $(Y_2, \Delta_2)$ is a log smooth pair with coefficients of $\Delta_2$ contained in $J$. Observe that if we continue blowing up this way, then this process will stop after a finitely many steps, since each step extracts an exceptional divisor $E_i$ over $X$ such that $a(E_i, X, \Delta)=a(E_i, X_i, \Delta_i)\<0$ and there are only finite many exceptional divisors over $X$ with this property. Assume that this process stabilizes at $g_n:X_n\to X_{n-1}$ for some $n\>1$. Rename $X_n$ by $X'$ and let $g:X'\to X$ be the composite of the all the morphisms $g_i, i=1,\ldots, n$. Write $K_{X'}+\Delta'=g^*(K_X+\Delta)$ and $\Delta'=\sum_{i=1}^{N'}d_iD_i'$. Then by our construction $(X', \Delta')$ is a log smooth pair such that $d_i\in J$ for all $i\in\{1,\ldots, N'\}$ and if $D_i'\cap D_j'\neq\emptyset$ for some $i, j\in\{1,\ldots, N'\}, i\neq j$, then $d_i+d_j-1<0$. Then from our claim in the second paragraph it follows that $(X', \Delta')$ has terminal singularity. This completes the proof.

\end{proof}

\begin{lemma}\label{lem:positive-dicrepancy}
	Let $(X, \Delta)$ be a log smooth pair of dimension $2$. Suppose that $\Delta=a_1D_1+a_2D_2+bD$, where $D_1, D_2$ and $D$ are prime Weil divisors, and $a_1, a_2, b$ are rational numbers such that $a_1, a_2<1$ and $b<0$. Assume that $D\cap D_1$ and $D\cap D_2$ are both non-empty. Let $p\in D$ be a closed point and $f:Y\to X$ is the blow up of $X$ at $p$. If $E$ is the exceptional divisor of $f$, then $a(E, X, \Delta)>0$. 
\end{lemma}

\begin{proof}
A simple computation shows that
\[
	a(E, X, \Delta)=\begin{cases}
		(1-a_1-b)>0 & \mbox{ if } p\in D\cap D_1,\\
		(1-a_2-b)>0 & \mbox{ if } p\in D\cap D_2,\\
		(1-b)>0 & \mbox{ if } p\in D\setminus(D_1\cup D_2).
	\end{cases}
\]
\end{proof}

\section{Main Theorem}
In this section we prove our main theorem.\\

\begin{proof}[Proof of Theorem \ref{thm:main-theorem}]
	First of all replacing $I$ by $I\cup \{1-\frac{1}{n}\; :\; n\in\mbN\}\cup\{1\}$ we may assume that $I$ contains the standard set $\{1-\frac{1}{n}\; :\; n\in\mbN\}\cup\{1\}$. Then replacing $(X, \Delta)$ by a dlt model, we may assume that $(X, \Delta)$ is dlt for all $(X, \Delta)\in\mfD$. Let $f:Y\to X$ be a log resolution such that all the exceptional divisors have discrepancies $>-1$. Write $K_Y+\Delta_Y=f^*(K_X+\Delta)$ and decompose $\Delta_Y=\Delta_Y^{=1}+\Delta_Y^{<1}$. Since $K_Y+\Delta_Y$ is big and being big an open property, there is an integer $n\gg 0$ such that $K_Y+\Delta'_Y$ is still big, where $\Delta'_Y:=(1-\frac{1}{n})\Delta_Y^{=1}+\Delta_Y^{<1}$. Moreover, note that $H^0(Y, \mcO_Y(\lrd m(K_Y+\Delta'_Y)\rrd))\subset H^0(Y, \mcO_Y(\lrd m(K_Y+\Delta_Y)\rrd))=H^0(X, \mcO_X(\lrd m(K_X+\Delta)\rrd))$ for all $m\>1$. Therefore replacing $(Y, \Delta'_Y)$ by $(X, \Delta)$ we may assume that $(X, \Delta)$ is klt for all $(X, \Delta)\in\mfD$.   
	Let $g:X'\to X$ be a log resolution of $(X, \Delta)$ and
	 \[
		K_{X'}+g^{-1}_*\Delta+\sum e_iE_i=g^*(K_X+\Delta).
	\] 
Since $(X, \Delta)$ is klt, $e_i<1$ for all $i$. So there is a postive integer $n>0$ such that $e_i<(1-\frac{1}{n})$ for all $i$.  Now define $\Delta':=g^{-1}_*\Delta+\sum (1-\frac{1}{n})E_i$. Then $H^0(X', \mcO_{X'}(\lrd m(K_{X'}+\Delta')\rrd))=H^0(X, \mcO_X(\lrd m(K_X+\Delta)\rrd))$ for all $m\>1$. Therefore replacing $(X, \Delta)$ by $(X', \Delta')$ we may assume that $(X, \Delta)$ is a log smooth pair. Now by \cite[Lemma 2.4]{HK19} (also see \cite[Theorem 4.6]{AM04} and \cite[Theorem 7.5]{Ale94})	there is a $0<\delta<1$ depending only on the set $I$ such that $K_X+(1-\delta)\Delta$ is big for all $(X, \Delta)\in\mfD$. Let $a>0$ be the minimum of the set $I\setminus\{0\}$ and $k:=\lru\frac{1}{a\delta}\rru$. Write $\Delta=\sum a_iD_i$ and define $a_i':=\frac{\lrd ka_i\rrd}{k}$ and $\Delta':=\sum a_i'D_i$. Then from Lemma \ref{lem:modifying-divisor} it follows that $(1-\delta)\Delta\<\Delta'\<\Delta$. Therefore $K_X+\Delta'$ is big and the coefficients of $\Delta'$ are contained in the finite set $J:=\{\frac{\ell}{k}: \ell=1, 2,\ldots, k-1\}$. Note that $H^0(X, \mcO_X(\lrd m(K_X+\Delta')\rrd))\subset H^0(X, \mcO_X(\lrd m(K_X+\Delta)\rrd))$ for all $m\>1$. Thus it is enough to prove the theorem for the pairs $(X, \Delta')$. Let $\mfD'$ be the collection of all such pairs $(X, \Delta')$. Now replacing $(X, \Delta')$ by its terminalization as in the Lemma \ref{lem:terminalization} we may assume that $(X, \Delta')$ is a log smooth terminal pair for all $(X, \Delta')\in\mfD'$. Note that since coefficients of $\Delta'$ are contained in a fixed finite set $J$ for all $(X, \Delta')\in\mfD'$, from \cite[Corollary 2.31(3)]{KM98} it follows that there is an $\ve>0$ depending only on the set $J$ (in particular on the set $I$) such that $(X, \Delta')$ is $\ve$-klt for all $(X, \Delta')\in\mfD'$. Then by Theorem \ref{thm:bounding-components} there is a positive integer $A(J, \ve)$ depending only on the set $J$ and $\ve$ (in particular on the set $I$) such that $\nt(\Delta'-\phi(\Delta'))\<A(J, \ve)$ for all pairs $(X, \Delta)\in\mfD'$. Therefore replacing $(X, \Delta'-\phi(\Delta'))$ by $(X, \Delta')$ we may assume that the number of components of $\Delta'$ is bounded from above for all $(X, \Delta')\in\mfD'$.\\
Now we run a $(K_X+\Delta')$-MMP and end with a minimal model $(X'', \Delta'')$, i.e. $K_{X''}+\Delta''$ is nef and big. Let $\mfD''$ be the collection of all such minimal models $(X'', \Delta'')$ for all $(X, \Delta')\in\mfD'$. We will show that the number of exceptional divisors over $X''$ with negative discrepancy is bounded above. To that end recall that $(X, \Delta')$ is terminal, so by Proposition \ref{pro:bounding-exceptional-divisors} the number of exceptional divisors over $X''$ with negative discrepancies with respect to $(X'', \Delta'')$ is bounded above by the number $A(J, \ve)$ defined above.\\
Now by Theorem \ref{thm:bounding-Cartier-index} there is a natural number $N$ depending only on the set $J$ and $\ve>0$ such that $N(K_{X''}+\Delta'')$ is Cartier for all $(X'', \Delta'')\in\mfD''$. Since $\vol(N(K_{X''}+\Delta''))\>\vol(K_{X''})$, by Lemma \ref{lem:birational-bound} the linear system $|K_{X''}+q(N(K_X''+\Delta''))|$ gives a birational map for all $q\>18$ and for all $(X'', \Delta'')\in\mfD''$. Replacing $N$ by $18N$ we may assume that $|K_{X''}+qN(K_{X''}+\Delta'')|$ is biratinal for all $q\>1$. Thus $|\lrd (qN+1)(K_{X''}+\Delta'')\rrd|$ gives a birational map for all $q\>1$. Then by Lemma \ref{lem:effective-bound} $|\lrd m(K_{X''}+\Delta'')\rrd|$ is birational for all $m\>(N^2+1)$ and for all $(X'', \Delta'')\in\mfD''$. This shows that $|\lrd m(K_X+\Delta')\rrd|$ gives a birational map for all $m\>m_0$. Consequently, $|\lrd m(K_X+\Delta)\rrd|$ gives a birational map for all $m\>(N^2+1)$ and for all $(X, \Delta)\in\mfD$, where $N$ depends only on the set $I$.
	
\end{proof}

\begin{lemma}\label{lem:effective-bound}
	Let $(X, \Delta\>0)$ be log pair and $N>0$ is a positive integer such that $N(K_X+\Delta)$ is Cartier. Assume that the linear system $|\lrd (qN+1)(K_X+\Delta)\rrd|$ gives a birational map onto its image for all $q\>1$. Then $|\lrd m(K_X+\Delta)\rrd|$ gives a birational map for all $m\>(N^2+1)$.
\end{lemma}

\begin{proof}
Set $q=N$ and choose a positive integer $k\>1$. Then by the division algorithm we have $k=(N+1)a+b$, where $0\<b\<N$. Therefore $N^2+k=(N^2+b)+a(N+1)$. Now if $b=0$, then $a\>1$ since $k>0$ and we can rewrite $N^2+k$ as $N^2+k=[(N+1)N+1]+(a-1)(N+1)$. Otherwise $b\>1$ and $N^2+b$ can be written as $N^2+b=(b-1)(N+1)+[(N-b+1)N+1]$. These calculations clearly show that $|\lrd m(K_X+\Delta)\rrd|$ gives a birational map for all $m\>(N^2+1)$. 
\end{proof}

\bibliographystyle{habbrv}
\bibliography{references.bib}

\end{document}